\documentclass[12pt, reqno]{amsart}
\usepackage[utf8]{inputenc}
\usepackage{amsmath}
\usepackage{amsfonts}
\usepackage{amssymb}
\usepackage{amsthm}
\usepackage{bm}
\usepackage{bbm}
\usepackage{xcolor}
\theoremstyle{plain}
\newtheorem{theorem}{Theorem}[section]
\newtheorem{lemma}[theorem]{Lemma}

\newtheorem{proposition}[theorem]{Proposition}
\newtheorem{remark}{Remark}
\newtheorem{conjecture}{Conjecture}

\title[E.G.G.S. and integral points on hyperelliptic curves]{A problem of Erd\H{o}s-Graham-Granville-Selfridge on integral points on hyperelliptic curves}
\author{Hung M. Bui, Kyle Pratt and Alexandru Zaharescu}
\address{Department of Mathematics, University of Manchester, Manchester M13 9PL, UK}
\email{hung.bui@manchester.ac.uk}
\address{All Souls College, Oxford OX1 4AL, UK}
\email{kyle.pratt@all-souls.ox.ac.uk}
\address{Department of Mathematics, University of Illinois at Urbana-Champaign, 1409 West Green Street, Urbana, IL 61801, USA
and Simion Stoilow Institute of Mathematics of the Romanian Academy, P.O. Box 1-764, RO-014700 Bucharest,
Romania}
\email{zaharesc@illinois.edu}


\allowdisplaybreaks

\begin{document}
\date{}

\maketitle

\begin{abstract}
Erd\H{o}s, Graham, and Selfridge considered, for each positive integer $n$, the least value of $t_n$ so that the integers $n+1, n+2, \dots, n+t_n $ contain a subset the product of whose members with $n$ is a square. An open problem posed by Granville concerns the size of $t_n$, under the assumption of the ABC Conjecture. We establish some results on the distribution of $t_n$, and in the process solve Granville's problem unconditionally.
\end{abstract}

\section{Introduction}

A question of Erd\H{o}s, Graham, and Selfridge (\cite{ES1992} and \cite[B30]{Guy2004}) asks to find the least
value of $t_n$ so that the integers $n+1, n+2, \dots, n+t_n $ contain a subset the product of whose
members with $n$ is a square. (If $n$ is a square then we set $t_n=0$.) That is, $t_n \geq 0$ is the least integer such that there are integers $1\leq j_1 < \cdots <j_s = t_n$ with 
\begin{align*}
n\prod_{i=1}^s (n+j_i) = \square,
\end{align*}
where ``$m = \square$'' means that the integer $m$ is a square. For instance, we easily compute that $t_2 = 4, t_3 = 5, t_5 = 5$, and $t_6 = 6$ since
\begin{align*}
2\cdot 3\cdot 6 &= 6^2\\
3 \cdot 6 \cdot 8 &= 12^2\\
5 \cdot 8 \cdot 10 &= 20^2\\
6 \cdot 8 \cdot 12 &= 24^2
\end{align*}
and none of the last numbers in the products can be replaced by smaller integers.

One can interpret $t_n$ in terms of integer points on hyperelliptic curves. As an example, we have $t_{14}=7$ from $14\cdot 15 \cdot 18\cdot 20\cdot 21=1260^2$, and this gives rise to the integral point $(x,y)=(14,1260)$ on the hyperelliptic curve $y^2=x(x+1)(x+4)(x+6)(x+7)$. Granville noted this connection \cite[B30]{Guy2004} and observed that effective versions of Faltings' Theorem \cite{Fal1983} would lead to corresponding effective bounds on $t_n$. Granville mentioned that the ABC conjecture should lead, via work of Elkies \cite{E1991} and Langevin \cite{L1993}, to stronger bounds on $t_n$. He also stated that, presumably, $t_n > n^c $ for some fixed constant $ c > 0 $, and perhaps one could prove this assuming the ABC conjecture.

Our first result below shows that this supposition fails dramatically. A beautiful result of Granville and Selfridge \cite[Corollary 1]{GS2001} shows that if the largest prime factor $P^+(n) $ of $n$ satisfies $P^+(n) > \sqrt{2n}+1$ then $t_n = P^+(n)$. This inspired us to study closely the relationship between $P^+(n)$ and $t_n$. While $t_n$ and $P^+(n)$ may no longer be close if $P^+(n) \leq \sqrt{2n}+1$, we find the distribution of $t_n$ continues to follow that of $P^+(n)$ in larger ranges.

\begin{theorem}\label{thm:cor of main thm}
For any fixed $c \in (0,1]$, 
$$
\lim_{x\to\infty} \frac{ \#\{ n \le x :  t_n \leq n^c \} }{x} = \lim_{x\to\infty} \frac{ \#\{ n \le x :  P^+(n) \leq n^c \} }{x} \; .
$$
\end{theorem}

\begin{remark}
The right-hand side in the statement of Theorem \ref{thm:cor of main thm} is equal to $\rho(1/c)$, where $\rho(u)$ is the well-known Dickman-de Bruijn function which plays a prominent role in the theory of smooth numbers \cite{HT1993}. Therefore, for every fixed $c > 0$ a positive proportion of integers $n$ satisfy $t_n \leq n^c$.
\end{remark}

Theorem \ref{thm:cor of main thm} is a corollary of a slightly stronger theorem (Theorem \ref{thm:dist of tn and P+n is same}) which we state and prove in Section \ref{sec:dist of tn} below.

Our next result shows there are integers $n$ attaining even smaller values of $t_n$, much smaller than $n^c$.

\begin{theorem}\label{thm:many n with small tn}
Let $\epsilon>0$ be fixed and sufficiently small, and let $x$ be sufficiently large depending on $\epsilon$. Then there are at least $x \exp \Big(-\big(\tfrac{3\sqrt{2}}{2}+\epsilon \big)\sqrt{\log x \log \log x} \Big)$ integers $n\leq x$ such that
\begin{align*}
t_n \leq \exp \left(\sqrt{(2+\epsilon)\log n \log \log n} \right).
\end{align*}
\end{theorem}

Suitable modification of the proof of Theorem \ref{thm:many n with small tn} shows there exist hyperelliptic curves of large genus which have integral points of large height.

\begin{theorem}\label{thm:hyper curve int point large height}
Fix a constant $c \in (0,1)$. There are arbitrarily large positive integers $J$ such that the following is true: there exist $N$ positive integers $1\leq j_1 < j_2 < \cdots < j_N < J$ with $N \geq J^{1-c}$ and a positive integer
\begin{align*}
x\geq \exp \left(\frac{c^2}{5}\frac{(\log J)^2}{\log \log J} \right)
\end{align*}
such that
\begin{align*}
x(x+J)\prod_{i=1}^N(x+j_i)
\end{align*}
is a square.
\end{theorem}

In the complementary direction, we prove a lower bound on $t_n$ when $n$ is not a square (recall $t_n=0$ when $n$ is a square). The proof also uses a framework of hyperelliptic curves, and requires bounding the height of integral points on curves.

\begin{theorem}\label{thm:lower bounds for tn}
If $n$ is a sufficiently large non-square integer, then
\begin{align*}
t_n \gg (\log \log n)^{6/5}(\log \log \log n)^{-1/5}.
\end{align*}
The implied constant is effectively computable.
\end{theorem}

The outline of the rest of the paper is as follows. In Section \ref{sec:notation} we describe the notation and conventions of the paper. In Section \ref{sec:dist of tn} we state Theorem \ref{thm:dist of tn and P+n is same}, of which Theorem \ref{thm:cor of main thm} is essentially a special case; we assemble the ingredients for the proof and then prove Theorems \ref{thm:dist of tn and P+n is same} and \ref{thm:cor of main thm}. In Section \ref{sec:many n with small tn} we prove Theorem \ref{thm:many n with small tn}. Section \ref{sec:large integral point} contains the results and modifications of the proof of Theorem \ref{thm:many n with small tn} necessary to prove Theorem \ref{thm:hyper curve int point large height}; we close the section with some comments and discussion. We prove Theorem \ref{thm:lower bounds for tn} in Section \ref{sec:lower bounds tn}. In the appendix we detail calculations allowing one to obtain strong bounds for the heights of integral points on hyperelliptic curves when the defining polynomial is monic of even degree. The results in the appendix serve only to motivate a conjecture which we state near the end of Section \ref{sec:lower bounds tn}.

\section{Notation and conventions}\label{sec:notation}

Given a positive integer $n$, the integer $t_n$ is the smallest nonnegative integer so that the integers $n+1, n+2, \dots, n+t_n $ contain a subset the product of whose members with $n$ is a square. If $n$ is a square then we define $t_n=0$.

The expression $m = \square$ means that the integer $m$ is a square.

A number $n$ is $y$--smooth if every prime divisor $p$ of $n$ satisfies $p\leq y$. We write $\Psi(x,y)$ for the number of $y$--smooth integers $n\leq x$, and $\rho(u)$ for the Dickman-de Bruijn function.

Given two finite sets $S$ and $T$, we write $S \Delta T$ for the symmetric difference of $S$ and $T$. That is, $S\Delta T$ consists of those elements which are in one of $S$ or $T$ but not both: $S\Delta T = (S \cup T) \backslash (S \cap T)$. By associativity one can consider the symmetric difference of any finite number of sets
\begin{align*}
S_1 \Delta S_2 \Delta \cdots  \Delta S_k = \{x : x \text{ is an element of an odd number of the sets }S_i\}.
\end{align*}

Given a finite set $S$ we write $\# S$ or $|S|$ for the cardinality of $S$. The power set of $S$, i.e. the set that consists of all the subsets of $S$, is denoted by $\mathcal{P}(S)$.

The finite field with two elements is denoted as $\mathbb{F}_2$.

We write $P^+(n)$ for the largest prime factor of a positive integer $n$. We set $P^+(1)=1$. We write $\omega(n)$ for the number of distinct prime factors of $n$.

The real number $x$ is always large. The notation $o(1)$ denotes a quantity tending to zero as some other parameter, usually $x$, tends to infinity. We write $f \ll g, g \gg f$, or $f = O(g)$ if there exists a constant $C$ such that $f \leq Cg$. We write $f \sim g$ if $f = (1+o(1))g$.

In discussion, but not in proofs, we sometimes refer to the \emph{height} of an integral point $(x,y)$ on a curve. By this we mean the naive height $\max(|x|,|y|)$. We similarly refer to the height of an integer polynomial, which is the maximum of the absolute value of its coefficients.

\section{The distribution of $t_n$: proof of Theorem \ref{thm:cor of main thm}}\label{sec:dist of tn}

As mentioned in the introduction, Theorem \ref{thm:cor of main thm} is a corollary of a somewhat stronger result, which we state here.

\begin{theorem}\label{thm:dist of tn and P+n is same}
Let $x$ be sufficiently large, and let $c$ satisfy
\begin{align*}
\frac{(\log \log \log x)^2}{\log \log x} \leq c \leq 1.
\end{align*}
Then
\begin{align*}
\sum_{\substack{n\leq x \\ t_n \leq x^c}}1 = \sum_{\substack{n\leq x \\ P^+(n) \leq x^c}} 1 + O \Big(\frac{x}{c\log x} \Big)
\end{align*}
uniformly in $c$.
\end{theorem}

\begin{remark}
The lower bound on $c$ could be relaxed slightly. The point is that, on this range of $c$, the error term $O(\frac{x}{c\log x})$ is smaller than the main term.
\end{remark}

We prove Theorem \ref{thm:dist of tn and P+n is same} by proving upper bounds for
\begin{align*}
\sum_{\substack{n\leq x \\ t_n \leq x^c}}1 \ \ \ \ \ \ \ \text{ and } \ \ \ \ \ \ \ \sum_{\substack{n\leq x \\ P^+(n) \leq x^c}}1
\end{align*}
in terms of each other. More precisely, the proof of Theorem \ref{thm:dist of tn and P+n is same} relies on the following two propositions.

\begin{proposition}[$t_n$ less than $P^+(n)$ on average]\label{prop:tn less than P+n}
Let $0 < c \leq 1$. Then
\begin{align*}
\sum_{\substack{n\leq x \\ t_n \leq x^c}}1 \leq \sum_{\substack{n\leq x \\ P^+(n) \leq x^c}}1 + O(x\exp(-\sqrt{\log x}))
\end{align*}
uniformly in $c$.
\end{proposition}

\begin{proposition}[$P^+(n)$ less than $t_n$ on average]\label{prop:P+n less than tn}
Let $0 < c \leq 1$. Then
\begin{align*}
\sum_{\substack{n\leq x \\ P^+(n) \leq x^c}}1 \leq \sum_{\substack{n\leq x \\ t_n \leq x^c}}1 + O \Big(\frac{x}{c\log x} \Big)
\end{align*}
uniformly in $c$.
\end{proposition}

\begin{proof}[Proof of Theorem \ref{thm:dist of tn and P+n is same} assuming Propositions \ref{prop:tn less than P+n} and \ref{prop:P+n less than tn}]
From Proposition \ref{prop:tn less than P+n} we have
\begin{align*}
\sum_{\substack{n\leq x \\ t_n \leq x^c}}1 \leq \sum_{\substack{n\leq x \\ P^+(n) \leq x^c}}1 + O\big(x\exp(-\sqrt{\log x})\big),
\end{align*}
and from Proposition \ref{prop:P+n less than tn} we have
\begin{align*}
\sum_{\substack{n\leq x \\ t_n \leq x^c}}1 \geq \sum_{\substack{n\leq x \\ P^+(n) \leq x^c}}1 - O \Big(\frac{x}{c\log x} \Big),
\end{align*}
so
\begin{align*}
\sum_{\substack{n\leq x \\ t_n \leq x^c}}1 = \sum_{\substack{n\leq x \\ P^+(n) \leq x^c}}1 + O \Big(\frac{x}{c\log x} \Big).
\end{align*}
This asymptotic formula is non-trivial provided
\begin{align*}
\sum_{\substack{n\leq x \\ P^+(n) \leq x^c}}1 > C \frac{x}{\log(x^c)}
\end{align*}
with $C$ a sufficiently large absolute constant. By \cite[Theorem 1]{Hil1986} we have
\begin{align*}
\Psi(x,x^c) \sim x \rho \left(1/c \right).
\end{align*}
Since $\rho(1/c) \gg_c 1$ the asymptotic is non-trivial for $c \geq \epsilon_0 > 0$, with $\epsilon_0 > 0$ sufficiently small and fixed. If $c < \epsilon_0$ and $\epsilon_0$ is sufficiently small then by \cite[Corollary 2.3]{HT1993} we have
\begin{align*}
\rho(1/c) \geq \exp\left( -c^{-1} \left(\log c^{-1} + \log \log c^{-1}\right)\right).
\end{align*}
By straightforward calculation we deduce that
\begin{align*}
\exp\left( -c^{-1} \left(\log c^{-1} + \log \log c^{-1}\right)\right) \geq (\log x)^{-0.1}
\end{align*}
if $c\geq \frac{(\log \log \log x)^2}{\log \log x}$, so the asymptotic formula is non-trivial in this range.
\end{proof}

We first turn our attention to Proposition \ref{prop:tn less than P+n}, since the proof is simpler and introduces some of the key ideas. The first result we need is a simple inequality relating $t_n$ and $P^+(n)$.

\begin{lemma}\label{lem:tn almost always bigger than P+n}
If $P^+(n)^2 \nmid n$, then $t_n \geq P^+(n)$.
\end{lemma}
\begin{proof}
Let $p = P^+(n)$. If $p^2 \nmid n$, then $n$ is not a square so by the definition of $t_n$ there are integers $1\leq j_1 < \cdots < j_s =t_n$ with
\begin{align*}
n\prod_{i=1}^s (n+j_i) = \square.
\end{align*}
Since $p$ divides the left-hand side to an even power but $p^2 \nmid n$ there is some $i$ such that $p \mid n+j_i$. Since $p\mid n$ and $p \mid n+j_i$ we have $p \mid j_i$, so $t_n \geq j_i \geq p$.
\end{proof}

We also need a result quantifying that it is rare for an integer $n$ to be divisible by the square of its largest prime factor (see also \cite[p. 345]{EG1976}).

\begin{lemma}[Bound for exceptional set with $P^+(n)^2 \mid n$]\label{lem:P+2 divides n is small set}
Let $\mathcal{E}$ denote the set of $n\leq x$ such that $P^+(n)^2 \mid n$. Then
\begin{align*}
|\mathcal{E}| &\ll x \exp(-\sqrt{\log x}).
\end{align*}
\end{lemma}
\begin{proof}
Any $n \in \mathcal{E}$ may be written as $n = p^2 m$, where $P^+(m) \leq p$, and therefore
\begin{align*}
|\mathcal{E}| &\leq \sum_{p\leq x^{1/2}} \sum_{\substack{m\leq x/p^2 \\ P^+(m) \leq p}} 1.
\end{align*}
We introduce a parameter $10\leq P \leq x^{1/2}$ and split the sum over $p$ at $P$. The contribution from $p > P$ is
\begin{align*}
&\leq \sum_{P < p \leq x^{1/2}} \sum_{m\leq x/p^2}1 \ll \sum_{P < p \leq x^{1/2}} \left( \frac{x}{p^2} + 1\right) \ll \frac{x}{P} + x^{1/2} \ll \frac{x}{P}.
\end{align*}
We bound the contribution from $p\leq P$ using Rankin's trick. Set $\alpha = 1 - \frac{1}{\log P}$, so that
\begin{align*}
\sum_{p\leq P}\sum_{\substack{m\leq x/p^2 \\ P^+(m) \leq p}} 1 &\leq x^{\alpha}\sum_{p\leq P} \frac{1}{p^{2\alpha}} \sum_{P^+(m) \leq p} \frac{1}{m^\alpha} \ll x^\alpha \sum_{p\leq P} \frac{1}{p^2} \prod_{q\leq p}\left(1 + \frac{3}{q}\right) \\
&\ll x^\alpha \sum_{p\leq P} \frac{(\log p)^3}{p^2} \ll x^\alpha = x \exp \Big( - \frac{\log x}{\log P}\Big).
\end{align*}
We have therefore proved
\begin{align*}
|\mathcal{E}| &\ll\frac{x}{P} + x \exp \Big( - \frac{\log x}{\log P}\Big),
\end{align*}
and the optimal choice is to take $P = \exp (\sqrt{\log x})$.
\end{proof}

We now have the tools to prove Proposition \ref{prop:tn less than P+n}.

\begin{proof}[Proof of Proposition \ref{prop:tn less than P+n}]
We split the integers $n\leq x$ according to whether or not $P^+(n)^2 \mid n$, so that
\begin{align*}
\sum_{\substack{n\leq x \\ t_n \leq x^c}} 1 &= \sum_{\substack{n\leq x \\ t_n \leq x^c \\ P^+(n)^2 \nmid n}} 1 + \sum_{\substack{n\leq x \\ t_n \leq x^c \\ P^+(n)^2 \mid n}} 1 \leq \sum_{\substack{n\leq x \\ t_n \leq x^c \\ P^+(n)^2 \nmid n}} 1 + |\mathcal{E}|,
\end{align*}
where $\mathcal{E}$ is the set in Lemma \ref{lem:P+2 divides n is small set}. If $P^+(n)^2 \nmid n$ then we have $P^+(n) \leq t_n$ by Lemma \ref{lem:tn almost always bigger than P+n}, so
\begin{align*}
\sum_{\substack{n\leq x \\ t_n \leq x^c \\ P^+(n)^2 \nmid n}} 1 = \sum_{\substack{n\leq x \\ t_n \leq x^c \\ P^+(n)\leq x^c \\ P^+(n)^2 \nmid n}} 1\leq \sum_{\substack{n\leq x \\ P^+(n) \leq x^c}} 1,
\end{align*}
by positivity. Therefore
\begin{align*}
\sum_{\substack{n\leq x \\ t_n \leq x^c}} 1 \leq \sum_{\substack{n\leq x \\ P^+(n) \leq x^c}} 1 + |\mathcal{E}|,
\end{align*}
and we finish with an appeal to Lemma \ref{lem:P+2 divides n is small set}.
\end{proof}

The proof of Proposition \ref{prop:P+n less than tn} is a little more circuitous, and relies upon an analysis of the number of subsets $S$ of an interval with $\prod_{n \in S} n = \square$. For this we require two more lemmas, though the reader may wish to skip ahead and see how the lemmas are used in establishing Proposition \ref{prop:P+n less than tn} before examining their proofs.

\begin{lemma}[Subset squares and $t_n$ in intervals]\label{lem:num subsets n plus tn < x plus y}
Let $I = (x,x+y]$ be an interval. The number of subsets $S$ of $I \cap \mathbb{N}$ such that $\prod_{n \in S} n = \square$ is equal to $2^B$, where
\begin{align*}
B &:= \#\{n > x : n+t_n \leq x+y\} = \#\{n \in I : n+t_n \in I\}.
\end{align*}
\end{lemma}
\begin{proof}
If $B=0$ then there is no nonempty subset $S$ of $I$ with $\prod_{n \in S} n = \square$. Since $2^0 = 1$ the lemma is true in this case (the empty product is 1), and we may therefore assume $B\geq 1$. 

Denote the integers $n \in I$ with $n+t_n \in I$ by $n_1 < \cdots < n_B$, with $B\geq 1$. By definition, for any such $n_i$ we have $n_i (n_i + t_{n_i}) \prod_{j=1}^{s_i} (n+k_{i,j}) = \square$ for some integers $1\leq k_{i,1} < \cdots < k_{i,s_i} < t_{n_i}$ (there might be more than one choice of integers $k_{i,j}$ which works, and if this is the case we arbitrarily choose one sequence $k_{i,1} < \cdots < k_{i,s_i}$ to associate to $n_i$). The subset of $S_i \subset I \cap \mathbb{N}$ associated with $n_i$ is $\{n_i,n_i+k_{i,1},\ldots,n_i+k_{i,s_i},n_i + t_{n_i}\}$.

Given any subset $T \subset \{1,\ldots,B\}$, the integer $\prod_{i \in T}\prod_{n \in S_i} n$ is a square, since it is the product of the squares $\prod_{n \in S_i} n$. By keeping track of how often an integer $m\in I$ appears among the different $S_i$ we may write
\begin{align*}
\prod_{i \in T}\prod_{n \in S_i} n = \prod_{m \in I} m^{e_m},
\end{align*}
where $e_m$ is a nonnegative integer. We then consider the parity of $e_m$ and note that
\begin{align*}
\prod_{\substack{m \in I \\ e_m \text{ odd}}} m = \square.
\end{align*}
We define a map $G: \mathcal{P}( \{1,\ldots,B\}) \rightarrow \mathcal{P}( I \cap \mathbb{N})$ by setting $G(T) = \{m \in I : e_m \text{ odd}\}$ as above. Note that $G(T)$ is the symmetric difference of the sets $S_i, i \in T$.

We claim that the map $G$ is injective. Let $T,T' \subset \{1,\ldots,B\}$ be two distinct subsets. Since $T \neq T'$, there is a least integer $1\leq k \leq B$ such that $k \in T \Delta T'$ (i.e. $k$ is in one of $T$ or $T'$ but not both). Without loss of generality we may assume $k \in T$ and $k \not \in T'$. We write
\begin{align*}
\prod_{i \in T}\prod_{n \in S_i} n = \prod_{m \in I} m^{e_m}, \ \ \ \ \ \ \prod_{i \in T'}\prod_{n \in S_i} n = \prod_{m \in I} m^{f_m},
\end{align*}
with $e_m,f_m$ nonnegative integers. Since $T\cap \{1,\ldots,k-1\} = T' \cap \{1,\ldots,k-1\}$, it follows that
\begin{align*}
\prod_{i \in T \cap \{k,\ldots,B\}}\prod_{n \in S_i} n = \prod_{m \in I} m^{e_m-a_m}, \ \ \ \ \ \ \ \prod_{i \in T' \cap \{k+1,\ldots,B\}}\prod_{n \in S_i} n = \prod_{m \in I} m^{f_m-a_m},
\end{align*}
where the $a_m$ are nonnegative integers. Since $n_k \in \bigcup_{j=k}^B S_j$ and $n_k \not \in \bigcup_{j=k+1}^B S_j$, we see that $e_{n_k}-a_{n_k} = 1$ and $f_{n_k}-a_{n_k} = 0$. Hence $e_{n_k}$ and $f_{n_k}$ have opposite parity, so $n_k \in G(T) \Delta G(T')$ and $G(T) \neq G(T')$. It follows that the map $G$ is injective.

Since the map $G$ is injective, each of the $2^B$ subsets $T$ of $\{1,\ldots,B\}$ has $\prod_{n \in G(T)} n = \square$. To complete the proof, we must show that if $S \subset I \cap \mathbb{N}$ with $\prod_{n \in S} n = \square$, then $S = G(T)$ for some $T \subset \{1,\ldots,B\}$. If $S = \varnothing$ then $S = G(\varnothing)$, so we may assume $S \neq \varnothing$. We write $S = \{b_1,\ldots,b_r\}$ with $x < b_1 < \cdots < b_r \leq x+y$. Since $\prod_{i=1}^r b_i = \square$ we have $b_1 + t_{b_1} \leq b_r \leq x+y$ by the definition of $t_{b_1}$. It follows that $b_1 = n_{u_1}$ for some $u_1\in \{1,\ldots,B\}$. 

Let $S_{u_1}$ be the subset of $I$ associated with $n_{u_1}$. Since $n_{u_1}$ is the least element of $S$ and $S_{u_1}$, and since the sets $S_i$ only contain integers $\geq n_i$, we see that either $S \Delta S_{u_1} = \varnothing$, or else $S \Delta S_{u_1} \neq \varnothing$ and the least element of $S \Delta S_{u_1}$ is $\geq n_{u_1} + 1$. If $S \Delta S_{u_1}$ is nonempty then we may replace repeat this process, with $S$ replaced by $S \Delta S_{u_1}$. We see this process must eventually terminate in the empty set, since the least element of the sets $S,S\Delta S_{u_1},S \Delta S_{u_1} \Delta S_{u_2},\ldots$ is a strictly increasing sequence of integers $\leq x+y$. We therefore have $S = S_{u_1} \Delta \cdots \Delta S_{u_k}$ with $1\leq u_1 < \cdots < u_k \leq B$, but then by definition we have $S = G(\{u_1,\ldots,u_k\})$.
\end{proof}

\begin{lemma}[$t_n$ in intervals and smooth numbers]\label{lem:tn more than y smooth minus pi y}
Let $I = (x,x+y]$ be an interval. We have
\begin{align*}
\#\{n > x : n+t_n \leq x+y\} \geq \#\{y-\text{smooth integers in } I\} - \pi(y).
\end{align*}
\end{lemma}
\begin{proof}
The idea is to construct many different subsets with the product of elements in the subset equal to a square, and utilize smooth numbers and basic linear algebra over $\mathbb{F}_2$ to this end.

The lemma is trivially true if $\#\{y-\text{smooth integers in } I\} \leq \pi(y)$. We may therefore assume $\#\{y-\text{smooth integers in } I\} > \pi(y) \geq 1$. Let $n_1 < \cdots < n_M$ be the $y$--smooth integers in $I$, where $M\geq 1+\pi(y)$. Each integer $n_i$ may be factored
\begin{align*}
n_i = \prod_{p\leq y} p^{e_{p,i}},
\end{align*}
where $e_{i,p}$ is a nonnegative integer. We reduce the exponents $e_{i,p}$ modulo 2 and form a $\pi(y) \times M$ matrix $\mathcal{M}=(e_{p,i}\ (\text{mod}\ 2))_{\substack{p\leq y \\ i\leq M}}$. We consider the entries of the matrix as lying in the field $\mathbb{F}_2$. Let $r$ denote the rank of $\mathcal{M}$, and observe that $1\leq r\leq \pi(y)$.

Let $V_1,\ldots,V_M$ be the column vectors of $\mathcal{M}$. We let $W$ denote the vector space spanned by the column vectors $V_i$, and note that $W$ has dimension $r$. As $W$ is spanned by any $r$ of the vectors $V_i$ which are linearly independent over $\mathbb{F}_2$, we may reorder and relabel the columns to assume that $W$ is spanned by $V_1,\ldots,V_r$.

Given any subset $S \subset \{r+1,\ldots,M\}$, we have $\sum_{i \in S} V_i \in W$. This vector has a unique representation
\begin{align*}
\sum_{i \in S} V_i = \sum_{k=1}^r c_{k,S} V_k, \ \ \ \ \ \ c_{k,S} \in \mathbb{F}_2.
\end{align*}
If we write $T_S = \{1\leq k \leq r : c_{k,S}\neq 0\}$ then we see that for every subset $S \subset \{r+1,\ldots,M\}$ there is a subset $T_S \subset \{1,\ldots,r\}$ such that
\begin{align*}
\prod_{i \in S \cup T_S} n_i = \square.
\end{align*}
There are $\geq 2^{M-r}$ such subsets $S$, and since $S\cup T_S \neq S' \cup T_{S'}$ if $S \neq S'$ we see there are $\geq 2^{M-r}$ subsets $U$ of $\{1,\ldots,M\}$ such that $\prod_{i \in U} n_i = \square$. By Lemma \ref{lem:num subsets n plus tn < x plus y} we have
\begin{align*}
2^{\#\{n > x : n+t_n \leq x+y\}}\geq 2^{M-r},
\end{align*}
and therefore
\begin{align*}
\#\{n > x : n+t_n \leq x+y\} &\geq M-r \geq M-\pi(y). \qedhere
\end{align*}
\end{proof}

\begin{proof}[Proof of Proposition \ref{prop:P+n less than tn}]
If $x$ is sufficiently large and $c> 0.51$, say, then \cite[Corollary 1]{GS2001} implies $t_n = P^+(n)$, and therefore
\begin{align*}
\sum_{\substack{n\leq x \\ P^+(n) \leq x^c}} 1 = \sum_{\substack{n\leq x \\ P^+(n) \leq x^{0.51}}} 1 + \sum_{\substack{n\leq x \\ x^{0.51}<t_n \leq x^c}} 1.
\end{align*}
The proposition in the case $c>0.51$ therefore follows from the proposition in the case $c\leq 0.51$, so we may assume $c\leq 0.51$. We set $Y = x^c$, so $P^+(n)\leq x^c$ is the same as $P^+(n)\leq Y$.

We partition the sum over $Y$-smooth integers as
\begin{align}\label{eq:bound P+ by small tn and large tn}
\sum_{\substack{n\leq x \\ P^+(n)\leq Y}} 1 &= \sum_{\substack{n\leq x \\ P^+(n)\leq Y \\ t_n\leq Y}} 1 + \sum_{\substack{n\leq x \\ P^+(n)\leq Y \\ t_n> Y}} 1 \leq \sum_{\substack{n\leq x \\ t_n\leq Y}} 1 + \sum_{\substack{n\leq x \\ P^+(n)\leq Y \\ t_n> Y}} 1.
\end{align}
The first of these sums is the main term, and we must show that the second is an error term.

We split into short intervals of length $Y$. Write $I_k = (kY,(k+1)Y]$, and note that $I_k \cap \mathbb{N} \subset [1,x]$ for $k\leq x/Y-1$. Then
\begin{align*}
\sum_{\substack{n\leq x \\ P^+(n)\leq Y \\ t_n> Y}} 1 &\leq \sum_{0\leq k \leq x/Y-1} \ \sum_{\substack{n \in I_k \\ P^+(n) \leq Y \\ t_n > Y}} 1 + Y.
\end{align*}
We next remove ``exceptional'' intervals $I_k$ which contain many elements of $\mathcal{E}$, where $\mathcal{E}$ as in Lemma \ref{lem:P+2 divides n is small set} is the set of $n\leq x$ such that $P^+(n)^2 \mid n$. We partition the set of $I_k$ into sets $\mathcal{G}$ and $\mathcal{B}$ where $I_k \in \mathcal{G}$ if
\begin{align*}
\sum_{\substack{n\in I_k \cap \mathcal{E}}} 1 \leq \frac{Y}{\log x},
\end{align*}
and $I_k \in \mathcal{B}$ if the opposite inequality holds. By Lemma \ref{lem:P+2 divides n is small set}
\begin{align*}
\frac{Y}{\log x} |\mathcal{B}| &\leq \sum_{\substack{0\leq k \leq x/Y-1 \\ I_k \in \mathcal{B}}} \sum_{\substack{n \in I_k \cap \mathcal{E}}} 1 \leq |\mathcal{E}| \ll x\exp(-\sqrt{\log x}),
\end{align*}
so
\begin{align*}
|\mathcal{B}| &\ll \frac{x}{Y}(\log x)^{-100},
\end{align*}
say. Therefore
\begin{align}\label{eq:bound large tn in short intervals}
\sum_{\substack{n\leq x \\ P^+(n)\leq Y \\ t_n> Y}} 1 &\leq \sum_{\substack{0\leq k\leq x/Y-1 \\ I_k \in \mathcal{G}}}\sum_{\substack{n \in I_k \\ P^+(n) \leq Y \\ t_n > Y}} 1 + O(x(\log x)^{-100}),
\end{align}
since $Y \leq x^{0.51}$.

For any $k$ with $I_k \in \mathcal{G}$ we have from Lemma \ref{lem:tn more than y smooth minus pi y} that
\begin{align*}
\sum_{\substack{n \in I_k \\ P^+(n) \leq Y}}1  - \pi(Y) &\leq \sum_{\substack{n \in I_k \\ n+t_n \leq (k+1)Y}} 1.
\end{align*}
Since $n > kY$ the condition $n+t_n \leq (k+1)Y$ implies $t_n \leq Y$. We split according to whether $P^+(n)^2 \mid n$, and use the fact that $I_k \in \mathcal{G}$ to obtain
\begin{align*}
\sum_{\substack{n \in I_k \\ P^+(n) \leq Y}}1  - \pi(Y) &\leq \sum_{\substack{n \in I_k \\ t_n \leq Y \\ P^+(n)^2 \nmid n}} 1 + \sum_{n \in I_k \cap \mathcal{E}} 1 \leq \sum_{\substack{n \in I_k \\ t_n \leq Y \\ P^+(n)\leq Y}} 1 + \frac{Y}{\log x}.
\end{align*}
Since
\begin{align*}
\sum_{\substack{n \in I_k \\ t_n \leq Y \\ P^+(n)\leq Y}} 1 = \sum_{\substack{n \in I_k \\ P^+(n)\leq Y}} 1 - \sum_{\substack{n \in I_k \\ t_n > Y \\ P^+(n)\leq Y}} 1
\end{align*}
we see that for $I_k \in \mathcal{G}$ we have
\begin{align}\label{eq:y log y bound for large tn in short intervals}
\sum_{\substack{n \in I_k \\ t_n > Y \\ P^+(n)\leq Y}} 1 \leq \pi(Y) + \frac{Y}{\log x} \ll \frac{Y}{\log Y}.
\end{align}
Together \eqref{eq:bound large tn in short intervals} and \eqref{eq:y log y bound for large tn in short intervals} yield
\begin{align}\label{1001}
\sum_{\substack{n\leq x \\ P^+(n)\leq Y \\ t_n> Y}} 1 \ll \frac{x}{\log Y},
\end{align}
and by \eqref{eq:bound P+ by small tn and large tn} we obtain
\begin{align*}
\sum_{\substack{n\leq x \\ P^+(n) \leq Y}}1 \leq \sum_{\substack{n\leq x \\ t_n \leq Y}}1 + O \Big(\frac{x}{\log Y} \Big),
\end{align*}
as desired.
\end{proof}

\begin{proof}[Proof of Theorem \ref{thm:cor of main thm}]
By \cite[Corollary 1]{GS2001} we may assume $c\leq 0.51$. We have
\begin{align*}
\sum_{\substack{n\leq x \\ P^+(n) \leq x^c}} 1 \sim \sum_{\substack{n\leq x \\ t_n \leq x^c}} 1 \sim \sum_{\substack{x/\log x < n\leq x \\ t_n \leq x^c}} 1
\end{align*}
by Theorem \ref{thm:dist of tn and P+n is same} and trivial estimation. we split the sum acording to the size of $t_n$ so that
\begin{align*}
\sum_{\substack{x/\log x < n\leq x \\ t_n \leq x^c}} 1 &= \sum_{\substack{x/\log x < n\leq x \\ t_n \leq n^c}} 1 + \sum_{\substack{x/\log x < n\leq x \\ n^c < t_n \leq x^c}} 1 \sim \sum_{\substack{n\leq x \\ t_n \leq n^c}} 1 + O \Big(\sum_{\substack{n\leq x \\ (x/\log x)^c < t_n \leq x^c}} 1 \Big).
\end{align*}
We must show
\begin{align*}
\sum_{\substack{n\leq x \\ (x/\log x)^c < t_n \leq x^c}} 1 = o(x).
\end{align*}

We split the sum over $n$ according to whether or not $n \in \mathcal{E}$, with $\mathcal{E}$ as in Lemma \ref{lem:P+2 divides n is small set}. The size of $\mathcal{E}$ is $o(x)$. If $n \not \in \mathcal{E}$ then $P^+(n) \leq t_n$, and we may further split the sum with $n \not \in \mathcal{E}$ according to whether or not $P^+(n) \leq (x/\log x)^c$. Hence
\begin{align*}
\sum_{\substack{n\leq x \\ (x/\log x)^c < t_n \leq x^c}} 1 &= \sum_{\substack{n\leq x \\ n \not \in \mathcal{E} \\ (x/\log x)^c < t_n \leq x^c}} 1+o(x)= \sum_{\substack{n\leq x \\ n \not \in \mathcal{E} \\ (x/\log x)^c < t_n \leq x^c \\ P^+(n) \leq (x/\log x)^c}} 1 + O \Big(\sum_{\substack{n\leq x \\ (x/\log x)^c < P^+(n) \leq x^c}} 1 \Big)+o(x).
\end{align*}
The $O$-term is easily bounded, since
\begin{align*}
\sum_{\substack{n\leq x \\ (x/\log x)^c < P^+(n) \leq x^c}} 1 \leq \sum_{n\leq x} \sum_{\substack{(x/\log x)^c < p \leq x^c\\p\mid n}}1 \ll x\sum_{(x/\log x)^c < p \leq x^c} \frac{1}{p} \ll x\frac{\log \log x}{\log x},
\end{align*}
the last inequality following from Mertens' theorem. For the other sum, we set $Y = (x/\log x)^c$ and note that
\begin{align*}
\sum_{\substack{n\leq x \\ n \not \in \mathcal{E} \\ (x/\log x)^c < t_n \leq x^c \\ P^+(n) \leq (x/\log x)^c}} 1 &\leq \sum_{\substack{n\leq x \\ t_n > Y \\ P^+(n) \leq Y}} 1,
\end{align*}
and this is $o(x)$ by \eqref{1001}.

We have therefore shown that
\begin{align*}
\sum_{\substack{n\leq x \\ t_n \leq n^c}} 1 \sim \sum_{\substack{n\leq x \\ P^+(n) \leq x^c}} 1. 
\end{align*}
Also,
\begin{align*}
\sum_{\substack{n\leq x \\ P^+(n) \leq x^c}} 1 \sim \sum_{\substack{x/\log x<n\leq x \\ P^+(n) \leq x^c}} 1 = \sum_{\substack{x/\log x<n\leq x \\ P^+(n) \leq n^c}} 1 + \sum_{\substack{x/\log x<n\leq x \\ n^c < P^+(n) \leq x^c}} 1 \sim \sum_{\substack{n\leq x \\ P^+(n) \leq n^c}} 1 + O\Big(\sum_{\substack{n\leq x \\ (x/\log x)^c < P^+(n) \leq x^c}} 1 \Big),
\end{align*}
and this last error was already shown to be $o(x)$, which completes the proof.
\end{proof}

\section{Small values of $t_n$: Proof of Theorem \ref{thm:many n with small tn}}\label{sec:many n with small tn}

The proof of Theorem \ref{thm:many n with small tn} uses estimates for smooth numbers and some elementary combinatorics. We introduce parameters $y < L \leq x^{o(1)}$, and the first idea is to find many short intervals $I \subset [1,x]$ of length $L$ which contain roughly the expected number of $y$--smooth numbers.

\begin{lemma}[Many intervals with expected number of smooths]\label{lem:many ints with smooths}
Let $x$ be sufficiently large, and let $y < L \leq x^{o(1)}$ with $y \geq \exp ((\log x)^{1/100})$. Then there are $\gg \Psi(x,y)L^{-1}$ disjoint intervals $I \subset [x/\log x,x]$ of length $L$ such that
\begin{align*}
\# \{y-\text{smooth integers in } I\} \gg L \frac{\Psi(x,y)}{x}.
\end{align*}
\end{lemma}
\begin{proof}
With $y$ as in the statement of the lemma we have by \cite[Theorem 1]{Hil1986} that
\begin{align*}
\Psi(x,y) \sim x \rho \Big(\frac{\log x}{\log y} \Big).
\end{align*}
We similarly have
\begin{align*}
\Psi(x/\log x,y) \sim \frac{x}{\log x} \rho \Big(\frac{\log x}{\log y} - \frac{\log \log x}{\log y} \Big) \sim \frac{x}{\log x}\rho \Big(\frac{\log x}{\log y} \Big),
\end{align*}
the last asymptotic following by the continuity of $\rho$. It follows that the number of smooth numbers in $(x/\log x,x]$ is
\begin{align*}
\geq (1-o(1)) \Psi(x,y).
\end{align*}

For $k$ a positive integer write $I_k = (kL,(k+1)L]$, and let $\mathcal{I}$ denote those $I_k$ which are contained in $(x/\log x,x]$. By trivial estimation $\# \mathcal{I} \asymp x/L$, and
\begin{align*}
\sum_{I \in \mathcal{I}} \# \{y-\text{smooth integers in } I\}\geq (1-o(1)) \Psi(x,y).
\end{align*}
Let $\delta > 0$ be a sufficiently small positive constant, and write $\mathcal{I} = \mathcal{B} \cup \mathcal{G}$, where $I \in \mathcal{B}$ if $\# \{y-\text{smooth integers in } I\} \leq \delta L \Psi(x,y)/x$, and $I \in \mathcal{G}$ if $\# \{y-\text{smooth integers in } I\} > \delta L \Psi(x,y)/x$. Then
\begin{align*}
(1-o(1)) \Psi(x,y) &\leq \sum_{I \in \mathcal{B}} \delta L \frac{\Psi(x,y)}{x} + \sum_{I \in \mathcal{G}} \# \{y-\text{smooth integers in } I\} \\
&\leq O \left(\delta \Psi(x,y) \right) + \sum_{I \in \mathcal{G}} \# \{y-\text{smooth integers in } I\}.
\end{align*}
If $\delta$ is sufficiently small then
\begin{align*}
\Psi(x,y) \ll \sum_{I \in \mathcal{G}} \# \{y-\text{smooth integers in } I\},
\end{align*}
where the implied constant is absolute. Since trivially $\# \{y-\text{smooth integers in } I\} \leq L$ we find that
\begin{align*}
\#\mathcal{G} &\gg \frac{\Psi(x,y)}{L}. \qedhere
\end{align*}
\end{proof}

If a short interval contains sufficiently many $y$--smooth numbers, then we can construct an integer $n$ with small $t_n$.

\begin{lemma}[Build small $t_n$ with many smooths]\label{lem:many smooths gives small tn}
Let $I$ be an interval of length $L$ such that
\begin{align*}
\#\{y-\text{smooth integers in } I\} > \pi(y).
\end{align*}
Then there exists an integer $n \in I$ with $t_n \leq L$.
\end{lemma}
\begin{proof}
The basic idea is somewhat similar to that of the proof of Lemma \ref{lem:tn more than y smooth minus pi y}.

Let $p_1 < p_2 < \cdots < p_R\leq y$ be the primes $\leq y$, so that $\pi(y) = R$. A $y$--smooth integer $n$ may be written as
\begin{align*}
n = \prod_{i=1}^R p_i^{e_i},
\end{align*}
where $e_i$ is a nonnegative integer. By considering only the parity of $e_i$ we obtain a map $\theta: \{y-\text{smooth integers}\} \rightarrow \mathbb{F}_2^{R}$ given by
\begin{align}\label{eq:defn of theta}
\theta(n) = (e_1\ (\text{mod}\ 2),\ldots, e_R\ (\text{mod}\ 2)).
\end{align}

Now let $m_1,\ldots,m_M$ be the $y$--smooth integers in $I$, and note that, by assumption, we have $M > R$. Given $J \subset \{1,\ldots,M\}$, let
\begin{align*}
n_J = \prod_{j \in J} m_j,
\end{align*}
and observe that $n_J$ is a $y$--smooth integer. Since $M > R$ the number $2^M$ of subsets of $\{1,\ldots,M\}$ is strictly greater than $2^R = \#\mathbb{F}_2^R$, so by the pigeonhole principle there exist distinct subsets $J,J'$ of $\{1,\ldots,M\}$ such that $\theta(n_J) = \theta(n_{J'})$. By the definition of $\theta$ this implies
\begin{align*}
\prod_{m \in J}m \cdot \prod_{m \in J'} m = \square.
\end{align*}
Note that
\begin{align*}
\prod_{m \in J}m \cdot \prod_{m \in J'} m = \prod_{m \in J \Delta J'}m \cdot \prod_{m \in J \cap J'} m^2,
\end{align*}
where $J \Delta J'$ is the symmetric difference of the sets $J$ and $J'$. Since $J \neq J'$ we see that $J \Delta J' \neq \varnothing$ and
\begin{align*}
\prod_{m \in J \Delta J'} m = \square.
\end{align*}
The least element $n$ of $J \Delta J'$ is then the desired integer.
\end{proof}

\begin{proof}[Proof of Theorem \ref{thm:many n with small tn}]
We define $y = \exp(\frac{\sqrt{2}}{2} \sqrt{\log x \log \log x})$ and
\begin{align*}
L = \exp \left( \left(\sqrt{2} + (\log \log x)^{-1/2} \right)\sqrt{\log x \log \log x}\right).
\end{align*}
By \cite[Theorem 1]{Hil1986} and Lemma \ref{lem:many ints with smooths} there are $\gg x L^{-1} \rho \big( \frac{\log x}{\log y}\big)$ intervals $I \subset [x/\log x,x]$ of length $L$ such that each interval $I$ contains
\begin{align*}
\gg L \rho \Big( \frac{\log x}{\log y}\Big)
\end{align*}
$y$--smooth numbers. By \cite[Corollary 2.3]{HT1993} we have
\begin{align*}
\rho(u) \geq \exp\left(-u \left(\log u + \log \log u\right)\right)
\end{align*}
provided $u\geq 1$ is sufficiently large, and therefore
\begin{align*}
\rho \Big( \frac{\log x}{\log y}\Big) \geq \exp\bigg(-\Big(\frac{\sqrt{2}}{2} + O \Big(\frac{\log \log \log x}{\log \log x} \Big) \Big)\sqrt{\log x \log \log x} \bigg).
\end{align*}
The number of $y$--smooth integers in each interval $I$ of length $L$ is therefore
\begin{align*}
&\gg \exp\bigg(\Big(\frac{\sqrt{2}}{2} + (\log \log x)^{-1/2} - O \Big(\frac{\log \log \log x}{\log \log x} \Big) \Big)\sqrt{\log x \log \log x} \bigg) \\
&\gg y \exp \left((\log x)^{1/2} (\log \log x)^{1/4} \right),
\end{align*}
hence the number of $y$--smooth integers in each interval is $> \pi(y) \sim \frac{y}{\log y}$. It follows from Lemma \ref{lem:many smooths gives small tn} that each interval $I$ contains an integer $n$ with
\begin{align*}
t_n &\leq \exp \left( \left(\sqrt{2} + (\log \log x)^{-1/2} \right)\sqrt{\log x \log \log x}\right).
\end{align*}
Since $n \in [x/\log x,x]$ we see $x\leq n(\log n)^2$, and therefore
\begin{align*}
t_n &\leq \exp \left( \left(\sqrt{2} + 2(\log \log x)^{-1/2} \right)\sqrt{\log n \log \log n}\right),
\end{align*}
as desired.
\end{proof}

We remark that our proof of Theorem \ref{thm:many n with small tn} has some similarities to heuristic run-time analysis of factoring algorithms \cite[p. 1477]{Pom1996} (see also \cite{CGPT2012} and \cite[Section 1]{GS2001}).

\section{Large integral points on hyperelliptic curves: Proof of Theorem \ref{thm:hyper curve int point large height}}\label{sec:large integral point}

As mentioned in the introduction, the proof of Theorem \ref{thm:hyper curve int point large height} draws on ingredients in the proof of Theorem \ref{thm:many n with small tn}. We also need an  additional lemma, which provides for the existence of sets with large symmetric difference provided we have sufficiently many sets upon which to draw.

\begin{lemma}[Many subsets implies a large symmetric difference]\label{lem:many subsets implies large sym diff}
Let $N$ be large and let $S_1,\ldots,S_K$ be distinct subsets of $\{1,\ldots,N\}$. If $K \geq 2^Q$ with $Q \geq N^{1/100}$, then there exist $i\neq j$ such that
\begin{align*}
|S_i \Delta S_j |> \frac{1}{6} \frac{Q}{\log N}.
\end{align*}
\end{lemma}

\begin{remark}
Lemma \ref{lem:many subsets implies large sym diff} is not far from best possible, since the subsets $S_1,\ldots,S_K$ could be all the subsets of $\{1,\ldots,Q\}$.
\end{remark}

\begin{proof}
Let $A$ be an arbitrarily chosen (nonempty) subset $S_i$. For any other subset $S$, we may uniquely write $S$ as the disjoint union $S = S' \cup S_A$, where $S' \cap A = \varnothing$ and $S_A \subset A$. Observe that
\begin{align*}
A \Delta S = (A \backslash S) \cup (S \backslash A) = (A \backslash S_A) \cup S'.
\end{align*}
Let $\delta \geq \frac{1}{100}$ be a parameter. If $|A \Delta S |\leq \delta \frac{Q}{\log N}$, then $|S'| \leq \delta \frac{Q}{\log N}$ and $|A \backslash S_A| \leq \delta \frac{Q}{\log N}$.

The number of choices for the set $S'$ is
\begin{align*}
\leq \sum_{0\leq k \leq \delta Q/\log N} {N \choose k},
\end{align*}
and the number of choices for the set $S_A$ is
\begin{align*}
\leq \sum_{0\leq k \leq \delta Q/\log N} {{|A|}\choose {|A| - k}} \leq \sum_{0\leq k \leq \delta Q/\log N} {{|A|}\choose k}\leq \sum_{0\leq k \leq \delta Q/\log N} {N \choose k}.
\end{align*}
By the upper bound ${N \choose k} \leq (eN/k)^k$, which is valid for $k\geq 1$, we find
\begin{align*}
\sum_{0\leq k \leq \delta Q/\log N} {N \choose k} \leq 1 + (eN)^{\delta Q/\log N} \sum_{k\geq 1}k^{-k} \leq (3N)^{\delta Q/\log N}\leq \exp \left(2\delta Q \right).
\end{align*}

It follows that the total number of choices of subset $S$ such that $|A \Delta S| \leq \delta Q /\log N$ is $\leq \exp(4\delta Q) < 1.95^Q$, the last inequality following if we choose $\delta = \frac{1}{6}$, say. Since there are $K \geq 2^Q$ subsets, there must be some subset $S_i$ such that $|A \Delta S_i| > \frac{1}{6}Q/\log N$.
\end{proof}

\begin{proof}[Proof of Theorem \ref{thm:hyper curve int point large height}]
Let $x$ be a large integer, which we think of as tending to infinity. In particular, $x$ is sufficiently large compared to any fixed quantity like $c$. As in the proof of Theorem \ref{thm:many n with small tn}, we set our ``smoothness'' parameter
\begin{align*}
y = \exp \Big(\frac{\sqrt{2}}{2}\sqrt{\log x \log \log x} \Big).
\end{align*}
Given a constant $C > \sqrt{2}$, we also define a length parameter
\begin{align*}
L = \exp \Big(C\sqrt{\log x \log \log x} \Big).
\end{align*}

We may apply Lemma \ref{lem:many ints with smooths} to deduce the existence of many disjoint intervals $I \subset [x/\log x, x]$ of length $L$ such that the number of $y$--smooth integers in $I$ is $\gg L \Psi(x,y)/x$. We fix one such interval $I$, and note that by the argument of Theorem \ref{thm:hyper curve int point large height} the number of $y$--smooth integers in $I$ is
\begin{align*}
\geq \exp \bigg(\Big(C - \frac{\sqrt{2}}{2} - o(1) \Big)\sqrt{\log x \log \log x} \bigg).
\end{align*}
If we let $M$ denote the number of $y$--smooth integers in $I$, and $R = \pi(y)$, then we see that
\begin{align*}
M - R&\geq \exp \bigg(\Big(C - \frac{\sqrt{2}}{2} - o(1) \Big)\sqrt{\log x \log \log x} \bigg) - \exp \Big(\frac{\sqrt{2}}{2}\sqrt{\log x \log \log x} \Big) \\
&\geq \exp \bigg(\Big(C - \frac{\sqrt{2}}{2} - o(1) \Big)\sqrt{\log x \log \log x} \bigg)\geq L^{1 - \frac{\sqrt{2}}{2C} - o(1)}.
\end{align*}

Let $n_1,\ldots,n_M$ be the $y$--smooth integers in $I$. Given a subset $S\subset\{1,\ldots,M\}$ we may construct the $y$--smooth integer
\begin{align*}
n_S &= \prod_{s \in S} n_s,
\end{align*}
and then map $n_S$ to $\mathbb{F}_2^R$ using the map $\theta$ from \eqref{eq:defn of theta} in the proof of Lemma \ref{lem:many smooths gives small tn}. By the pigeonhole principle, there is some $v \in \mathbb{F}_2^R$ and $\geq 2^{M-R}$ subsets $S$ of $\{1,\ldots,M\}$ such that $\theta(n_S)=v$ for every such $S$. We apply Lemma \ref{lem:many subsets implies large sym diff} to obtain the existence of two subsets, call them $S$ and $T$, of $\{1,\ldots,M\}$ such that
\begin{align*}
|S \Delta T| &\geq \frac{1}{6} \frac{M-R}{\log M}\geq L^{1 - \frac{\sqrt{2}}{2C} - o(1)}.
\end{align*}

By construction we have
\begin{align*}
\prod_{s \in S \Delta T} n_s = \square.
\end{align*}
We may arrange the integers $n_s$ in increasing order and write them as $n,n+j_1,\ldots,n+j_V$ for some integer $n\in [x/\log x,x]$ and some integers $1\leq j_1 < j_2 < \cdots  < j_V \leq L$. It follows that
\begin{align*}
n(n+j_V)\prod_{i=1}^{V-1} (n+j_i) = \square.
\end{align*}

We claim that setting $J = j_V$ gives rise to a $J$ as in the statement of the theorem. First, note that $V-1\geq L^{1 - \frac{\sqrt{2}}{2C} - o(1)}-2 \geq L^{1 - \frac{\sqrt{2}}{2C} - o(1)}\geq J^{1-c}$, the last inequality holding if we set
\begin{align*}
C = \frac{\sqrt{2}}{c},
\end{align*}
say. Since $x\leq n(\log n)^2$ we see that
\begin{align*}
L =\exp \left(C\sqrt{\log x \log \log x} \right) \leq \exp \left((C+o(1))\sqrt{\log n \log \log n} \right).
\end{align*}
This implies
\begin{align*}
\frac{1-o(1)}{C^2} (\log L)^2 \leq \log n \log \log n,
\end{align*}
which in turn implies
\begin{align*}
\log n \geq \frac{1-o(1)}{2C^2} \frac{(\log L)^2}{\log \log L}.
\end{align*}
Recalling our choice for $C$ and that $J\leq L$ we find
\begin{align*}
n&\geq \exp \left(\frac{c^2}{5}\frac{(\log J)^2}{\log \log J} \right).
\end{align*}
Since every large $x$ gives rise to such a $J$, and since $J \geq L^{1/3}$, say, which tends to infinity with $x$, we may take $J$ to be arbitrarily large, as claimed.
\end{proof}

We close this section with some comments on Theorem \ref{thm:hyper curve int point large height}. In particular, it is worth comparing Theorem \ref{thm:hyper curve int point large height} with more trivial considerations.

First, we note that it is easy to obtain points $(x,y)$ on a hyperelliptic curve of large genus if we allow $x=0$ (so that $y$ is large). Indeed, consider the hyperelliptic curve $y^2 = P(x)= x^g + D^2$, where $g\geq 5$ is large and $D$ is a large positive integer. The point $(0,D)$ clearly lies on the hyperelliptic curve, and since the height $H$ of $P$ is $D^2$ we see the integral point $(0,D)$ has height $\gg H^{1/2}$. We might expect that all integral points on a hyperelliptic curve $y^2= P(x)$ have height $\ll H^{O(1)}$, so this trivial construction is already fairly sharp.

Second, we consider hyperelliptic curves with integral points $(x,y)$ where $x\neq 0$. The hyperelliptic curve
\begin{align*}
y^2 =P(x)= \prod_{i=1}^J (x+j), \ \ \ \ \ \ \ J \geq 5,
\end{align*}
is similar to the curves constructed in Theorem \ref{thm:hyper curve int point large height}, and this curve has the integral point $(-J,0)$. The polynomial $P$ has height $H=J!$, so the integral point $(-J,0)$ on the curve has height
\begin{align*}
&\gg \frac{\log H}{\log \log H}
\end{align*}
for large $J$. In contrast, Theorem \ref{thm:hyper curve int point large height} provides integral points $(x,y)$ on curves $y^2 = P(x)$ with $xy\neq 0$ and
\begin{align*}
x\gtrapprox (\log H)^{\frac{\log \log H}{\log \log \log H}} \gg_A (\log H)^A,
\end{align*}
where $H$ is the height of $P$. 

It would be very interesting to construct hyperelliptic curves of large genus having integral points $(x,y)$ with $x\geq H^c$, for $c>0$ some fixed constant.

\section{Lower bounds on $t_n$: Proof of Theorem \ref{thm:lower bounds for tn}}\label{sec:lower bounds tn}

If $n$ is a large non-square integer, then by definition
\begin{align*}
n(n+t_n)\prod_{i=1}^s (n+j_i) = \square,
\end{align*}
where $1\leq j_1 < \cdots < j_s < t_n$ are integers (obviously we must have $s < t_n$). If $t_n$ is very small compared to $n$, then the curve
\begin{align}\label{eq:hyperelliptic defn lower bound tn}
y^2 = x(x+J)\prod_{i=1}^s(x+j_i)
\end{align}
contains an integral point with $x$ extremely large (here and throughout the section we write $J = t_n$ in keeping with the notation of our other theorems). We rely on a uniform bound for the height of integral points on hyperelliptic curves due to B\'erczes, Evertse, and Gy\H{o}ry \cite{BEG2013}. Their method utilizes linear forms in logarithms.

We use different arguments depending on the size of $s$, with $s$ as in \eqref{eq:hyperelliptic defn lower bound tn}. When $s=0$ trivial arguments suffice to bound the size of $x$. If $s$ is at least one but is smaller than a small power of $J$, then we consider the hyperelliptic equation \eqref{eq:hyperelliptic defn lower bound tn} directly and apply the result of B\'erczes, Evertse, and Gy\H{o}ry. When $s$ is larger than a small power of $J$ it is more efficient to extract a suitable system of generalized Pell equations from \eqref{eq:hyperelliptic defn lower bound tn} and bound the size of solutions to these Pell equations. The coefficients of the Pell equations have size controlled by prime divisors $\leq J$, and we can use some elementary arguments to find a system with coefficients that are smaller than what a trivial bound would give. The final bound results from balancing the arguments coming from small $s$ and large $s$.

The following lemma handles the trivial case where $s=0$.

\begin{lemma}[Trivial case, $s=0$]\label{lem:s=0 lower bound tn}
Let $J\geq 1$ be an integer. If $x$ and $y$ are positive integers with $y^2 = x(x+J)$, then $x\leq J^2$.
\end{lemma}
\begin{proof}
If $x$ and $x+J$ have greatest common divisor $d\geq 1$, then $d \mid J$. We change variables $x = dz$ and find $(y/d)^2 = z(z+J/d)$, where $(z,z+J/d)=1$. Then $z=a^2$ and $z+J/d = b^2$ for some positive integers $b>a$. Then
\begin{align*}
J/d = b^2-a^2 = (b+a)(b-a) \geq b+a,
\end{align*}
so $a,b\leq J/d$. Then $z=a^2\leq J^2/d^2$ and $x=dz \leq J^2/d\leq J^2$.
\end{proof}

The next lemma is the theorem of B\'erczes, Evertse, and Gy\H{o}ry \cite{BEG2013} in the special case we require.

\begin{lemma}[Height of integral points]\label{lem:BEG bound on int points}
Let $P(x) = \sum_{i=0}^n a_i x^i \in \mathbb{Z}[x]$ with $\deg(P)\geq 3$ and no repeated roots. Write $\max_i |a_i| =H$. If $x$ and $y$ are positive integers with $y^2 = P(x)$ then
\begin{align*}
\max(\log x, \log y) &\leq (4n)^{212n^4}H^{50n^4}.
\end{align*}
\end{lemma}
\begin{proof}
This is \cite[Thereom 2.2]{BEG2013} with $K = \mathbb{Q}$ and $S$ equal to the infinite place of $\mathbb{Q}$, where $b=1$ in \cite[(2.5)]{BEG2013}.
\end{proof}

The following lemma is useful when $s$ is small.

\begin{lemma}[Bound on height when $s$ is small]\label{lem:lower bound tn small s}
Let $J\geq 2$ be an integer, and let $1\leq j_1 < \cdots < j_s < J$ be integers, where $s\geq 1$. If $x$ and $y$ are positive integers with $y^2 = x(x+J)\prod_{i=1}^s (x+j_i)$ then
\begin{align*}
\log x \leq \exp \big(O(s^5 \log J) \big).
\end{align*}
\end{lemma}
\begin{proof}
The integer polynomial $x(x+J)\prod_{i=1}^s (x+j_i)$ has degree $s+2$ and clearly has no repeated roots. The coefficients of the polynomial all have size $\leq J^{s+1}$, so by Lemma \ref{lem:BEG bound on int points} we see any solution to $y^2 = x(x+J)\prod_{i=1}^s (x+j_i)$ satisfies
\begin{align*}
\log x &\leq\big (4(s+2)\big)^{212(s+2)^4} (J^{s+1})^{50(s+2)^4}\leq \exp \big(O(s^5 \log J) \big).\qedhere
\end{align*}
\end{proof}

When $s$ is large we argue more carefully. We use the following lemma to control the coefficients of an auxiliary hyperelliptic equation.

\begin{lemma}[Finding numbers with fewer prime factors]\label{lem:find b's with few prime factors}
Let $J$ be a sufficiently large positive integer, and let $b_1,\ldots,b_t$ be positive integers all of whose prime factors are $\leq J$. If $100\leq t \leq J^{1/2}/\log J$ and $\textup{gcd}(b_i,b_j)\leq J$ for all $i\neq j$, then there exist distinct $b_i,b_j,b_k$ with
\begin{align*}
\omega(b_i),\omega(b_j),\omega(b_k) \ll \frac{J}{t\log J}.
\end{align*}
\end{lemma}
\begin{proof}
The goal is to improve upon the trivial bound $\omega(b_i) \ll J/\log J$ by a factor of $t$. We observe that, by the prime number theorem, any distinct $b_i$ and $b_j$ have $\ll \log J/\log \log J\leq \log J$ prime factors in common, since $\text{gcd}(b_i,b_j)\leq J$. Let $s_i$ denote the set of prime factors of $b_i$. Note that $|s_i \cap s_j| \leq \log J$ for any $i\neq j$, and that each $s_i$ is contained in the set of all primes $\leq J$.

Without loss of generality we may assume that $|s_1| \geq |s_2| \geq \cdots \geq |s_t|$. We claim that
\begin{align}\label{eq:lower bound on size of union}
|s_1 \cup \cdots \cup s_r| \geq r|s_r| - \frac{r(r-1)}{2}\log J
\end{align}
for each $1\leq r \leq t$. This inequality trivially holds for $r=1$, so suppose the inequality holds for $r$ and we wish to show it holds for $r+1$.

Let $A = s_1 \cup \cdots \cup s_r$, so that by inclusion-exclusion we have
\begin{align*}
|s_1 \cup \cdots \cup s_{r+1}| &= |A \cup s_{r+1}| = |A| + |s_{r+1}| - |A \cap s_{r+1}| \\
&\geq r|s_r| - \frac{r(r-1)}{2}\log J + |s_{r+1}| - \sum_{i=1}^r |s_i \cap s_{r+1}|,
\end{align*}
where in the second line we have used the induction hypothesis and
\begin{align*}
|A \cap s_{r+1}| = \bigg|\bigcup_{i=1}^r (s_i \cap s_{r+1}) \bigg|\leq \sum_{i=1}^r |s_i \cap s_{r+1}|.
\end{align*}
Since $|s_i \cap s_{r+1}| \leq \log J$ for every $i$ and $|s_r| \geq |s_{r+1}|$, we obtain
\begin{align*}
|A \cap s_{r+1}| \geq (r+1)|s_{r+1}| - \frac{r(r+1)}{2}\log J,
\end{align*}
as desired. This completes the proof of the claim.

Applying \eqref{eq:lower bound on size of union} yields
\begin{align*}
|s_r| &\leq \frac{1}{r}|s_1 \cup \cdots \cup s_r| + \frac{r-1}{2}\log J
\end{align*}
for any $1\leq r \leq t$. Since each set $s_i$ is contained in the set of all primes $\leq J$ we have
\begin{align*}
|s_r| &\ll \frac{J}{r\log J} + r\log J,
\end{align*}
and since $r\leq t\leq J^{1/2}/\log J$ we have
\begin{align*}
|s_r| &\ll \frac{J}{r\log J}.
\end{align*}
We finish the proof by taking $i=t-2,j=t-1$, and $k=t$.
\end{proof}

We are now ready to obtain a bound when $s$ is large.

\begin{lemma}[Bound on height when $s$ is large]\label{lem:lower bound tn large s}
Let $J$ be a sufficiently large positive integer, and let $1\leq j_1 < \cdots < j_s < J$ be integers, where $J^{1/100}\leq s < J$. If $x$ and $y$ are positive integers with $y^2 = x(x+J)\prod_{i=1}^s (x+j_i)$ then
\begin{align*}
\log x &\leq \exp \bigg(O \Big(\frac{J}{t\log J} \Big) \bigg),
\end{align*}
where $t$ is any integer satisfying $J^{1/100} \leq t \leq \min (s, J^{1/2}/\log J)$.
\end{lemma}
\begin{proof}
We write $j_0 = 0$ and $j_{s+1} = J$, so that the hyperelliptic equation is
\begin{align*}
y^2 = \prod_{i=0}^{s+2} (x+j_i).
\end{align*}
We observe that if $d \mid( x+j_i)$ and $d\mid (x+j_{i'})$ then $d\mid |j_i - j_{i'}| \leq J$, so the greatest common divisor of any two distinct $x+j_i$ is $\leq J$. Therefore, we may uniquely write
\begin{align*}
x+j_i = a_i y_i^2,
\end{align*}
where $a_i$ is divisible only by primes $\leq J$ and $y_i$ is divisible only by primes $> J$. By pulling out square factors of $a_i$ we obtain the more convenient factorization
\begin{align*}
x+j_i = b_i z_i^2,
\end{align*}
where $b_i$ is squarefree and divisible only by primes $\leq J$. Observe that $\text{gcd}(b_i,b_{i'}) \leq J$ for $i\neq i'$.

Choose any $t$ of the $b_i$, with $t$ as in statement of the lemma. Then by Lemma \ref{lem:find b's with few prime factors} there exist three distinct $b_i,b_k$, and $b_\ell$ with $\omega(b_i),\omega(b_k),\omega(b_\ell) \ll \frac{J}{t\log J}$. From the equations
\begin{align*}
x+j_i &= b_i z_i^2, \ \ \ \ \  x+j_k = b_k z_k^2, \ \ \ \ \ x+j_\ell = b_\ell z_\ell^2,
\end{align*}
we deduce
\begin{align}\label{eq:quartic hyperell lower bound tn}
(b_kb_\ell z_kz_\ell)^2 &= b_kb_\ell(x+j_k)(x+j_\ell) = b_kb_\ell(b_iz_i^2+j_k-j_i)(b_iz_i^2+j_\ell-j_i).
\end{align}
The quartic polynomial $b_kb_\ell(b_iz_i^2+j_k-j_i)(b_iz_i^2+j_\ell-j_i)$ has no repeated roots (since $j_k \neq j_\ell$) and has coefficients with absolute value
\begin{align*}
\leq J^2 b_ib_kb_\ell \leq J^{2+\omega(b_i)+\omega(b_k)+\omega(b_\ell)}\leq \exp\big( O \left( J/t\right) \big).
\end{align*}
We apply Lemma \ref{lem:BEG bound on int points} to the hyperelliptic equation \eqref{eq:quartic hyperell lower bound tn} to find
\begin{align*}
\log z_i \leq 16^{212\cdot 4^4} \exp\big( O \left( J/t\right) \big)^{50\cdot 4^4}\ll \exp\big( O \left( J/t\right) \big).
\end{align*}
Since $x+j_i = b_i z_i^2$ this implies $\log x \leq \exp\left( O \left( J/t\right) \right)$.
\end{proof}

\begin{proof}[Proof of Theorem \ref{thm:lower bounds for tn}]
Let $n$ be a large non-square integer. Write $J = t_n$ so that by definition we have
\begin{align*}
y^2 = n(n+J) \prod_{i=1}^s (n+j_i)
\end{align*}
for some integers $1\leq j_i < \cdots < j_s< J$ and $s\geq 0$. If $s=0$ then Lemma \ref{lem:s=0 lower bound tn} implies $n\leq J^2$. We may therefore assume $s\geq 1$.

If $J$ is bounded then by Lemma \ref{lem:BEG bound on int points} we see $n$ is effectively bounded in terms of $J$, so we may assume $J$ is sufficiently large. We define $t = \lfloor (J/\log J)^{1/6}\rfloor$. If $1\leq s\leq t$ then Lemma \ref{lem:lower bound tn small s} gives
\begin{align*}
\log n &\leq \exp \big(O(t^5\log J)\big) = \exp \Big(O\big(J^{5/6}(\log J)^{1/6} \big) \Big).
\end{align*}
If $t\leq s < J$ then applying Lemma \ref{lem:lower bound tn large s} gives
\begin{align*}
\log n &\leq \exp \big(O(J/t)\big) = \exp \Big(O\big(J^{5/6}(\log J)^{1/6} \big) \Big).
\end{align*}
Therefore, in any case we have
\begin{align*}
\log n &\leq \exp \Big(O\big(J^{5/6}(\log J)^{1/6} \big) \Big),
\end{align*}
which implies
\begin{align*}
\log \log n \ll J^{5/6} (\log J)^{1/6}.
\end{align*}
This last inequality implies in turn that $J \gg (\log \log n)^{6/5}(\log \log \log n)^{-1/5}$.
\end{proof}

One can likely obtain improvements to Theorem \ref{thm:lower bounds for tn} by working with a version of Lemma \ref{lem:BEG bound on int points} which exploits particular features of the hyperelliptic equation \eqref{eq:hyperelliptic defn lower bound tn}. For example, it should be possible to take advantage of the fact that the polynomial $x(x+J)\prod_{i=1}^s (x+j_i)$ has rational roots.

In certain situations, it is possible to obtain a strong bound on the height of an integral point on a hyperelliptic curve $y^2=P(x)$ through elementary methods. This is possible, for instance, when the polynomial $P$ is monic of even degree, in which case Runge's method may apply (see \cite[Chapter 4]{Mas2016}). In our case of interest, we may obtain strong bounds on the height of an integral point on the hyperelliptic curve
\begin{align*}
y^2 = x(x+J)\prod_{i=1}^s (x+j_i)
\end{align*}
when $s$ is even. There are several results in the literature in this direction e.g. \cite{Le1995, Sza2002}, and we provide an alternative account in the appendix

By Theorem \ref{thm:good int point bounds for even degree}, if the product
\begin{align*}
n(n+t_n)\prod_{i=1}^s (n+j_i) = y^2
\end{align*}
has an even number of terms (i.e. if $s$ is even) then $n \ll t_n^{O(t_n)}$ and hence
\begin{align}\label{eq:good lower bound on tn}
t_n \gg \frac{\log n}{\log \log n}.
\end{align}
We might expect that the bound $n \ll t_n^{O(t_n)}$ holds if $s$ is odd as well, so that the lower bound \eqref{eq:good lower bound on tn} holds in all cases (provided that $n$ is not a square). We therefore make the following conjecture, which features a little breathing room compared to \eqref{eq:good lower bound on tn}.

\begin{conjecture}
Let $c \in (0,1)$ be a fixed constant, and assume $n$ is a non-square integer which is sufficiently large in terms of $c$. Then
\begin{align*}
t_n \geq (\log n)^{1-c}.
\end{align*}
\end{conjecture}

In making this conjecture we reason by analogy with conjectures for elliptic curves. Given a monic quartic polynomial $p(x)$ of height $H$, one may show, using methods similar to those of the appendix, that any integral point on the curve $y^2=p(x)$ has height $\ll H^{O(1)}$. The Hall-Lang conjecture (see \cite[Conjecture 5]{Lang1983}, also \cite[p. 1122]{Sta2016}) says the same bound should hold when $y^2=p(x)$ is an elliptic curve in minimal Weierstrass form (so that, in particular, $p(x)$ is a cubic).

\appendix

\section{Bounds for integral points on hyperelliptic curves of even degree}

We let $J$ be sufficiently large, and let $u\geq 2$ be an integer. Let $0 = j_1 <\cdots < j_{2u} = J$ be integers. We write
\begin{align}\label{eq:defn of poly P}
P(x) = \prod_{i=1}^{2u}(x+j_i).
\end{align}
We wish to bound the size of a positive integer $x$ which satisfies $P(x) = y^2$. There is no harm in assuming $u\geq 2$ since if $u=1$ we already have a satisfactory bound via Lemma \ref{lem:s=0 lower bound tn}.

We first show that one may write $P(x) = f(x)^2+g(x)$, where $f(x)$ is monic of degree $u$, $\deg(g) < \deg(f)$, and the coefficients of $f$ are rational numbers with denominators of controlled size.

\begin{lemma}[$P$ is close to a square]\label{lem:write P as f2 plus g}
Let $P(x)$ be given as in \eqref{eq:defn of poly P}. There exists a monic polynomial $f(x) = \sum_{i=0}^u a_i x^i \in \mathbb{Q}[x]$ with $|a_i| \leq (u-i)^{u-i}(uJ)^{u-i}$ and $4^{u-i} a_i \in \mathbb{Z}$, and a polynomial $g(x) = \sum_{i=0}^{u-1} b_i x^i \in \mathbb{Q}[x]$ with $|b_i| \leq 5u^{4u-2i}J^{2u-i}$ such that $P(x) = f(x)^2 + g(x)$.
\end{lemma}
\begin{proof}
It is not difficult to choose a polynomial $f$ such that $f^2$ matches the coefficients of $P$ down to order $x^u$, but it is somewhat tedious to prove bounds for the coefficients of $f$.

Write $P(x) = \sum_{i=0}^{2u} q_i x^i \in \mathbb{Z}[x]$ and note that $0\leq q_i \leq {{2u} \choose i} J^{2u-i}$ with $q_{2u}=1$. Now write $f(x) = \sum_{i=0}^u a_i x^i$, where the $a_i$ are rational numbers to be determined and $a_u=1$. Since 
\begin{align*}
f(x)^2 = \sum_{\ell=0}^{2u} x^\ell \sum_{\max(\ell-u,0)\leq i \leq \min(\ell,u)}a_i a_{\ell-i}
\end{align*}
we wish to choose the $a_i$ so that
\begin{align}\label{eq:ell coeff equation}
\sum_{\ell-u\leq i \leq u}a_i a_{\ell-i} = q_\ell
\end{align}
for $u\leq \ell \leq 2u$. We choose the coefficients $a_i$, beginning with $a_u$ and proceeding successively down to $a_0$. Actually, we have already chosen $a_u = 1$, so \eqref{eq:ell coeff equation} holds when $\ell=2u$.

Assume we have chosen the coefficients $a_u,\ldots,a_{u-r}$ for some $0\leq r \leq u-1$ so that \eqref{eq:ell coeff equation} holds for $2u-r \leq \ell \leq 2u$. When $\ell = 2u-r-1$ we have
\begin{align*}
\sum_{\ell-u\leq i \leq u}a_i a_{\ell-i} = 2a_{u-r-1} + \sum_{u-r\leq i \leq u-1}a_i a_{2u-r-1-i},
\end{align*}
so \eqref{eq:ell coeff equation} holds if
\begin{align*}
2a_{u-r-1} + \sum_{u-r\leq i \leq u-1}a_i a_{2u-r-1-i} = q_{2u-r-1}.
\end{align*}
With $i$ in the stated range we have $i,2u-r-1-i \geq u-r$, so the coefficient $a_{u-r-1}$ appears only in the first term. We may therefore choose
\begin{align*}
a_{u-r-1} = \frac{1}{2}\Big(q_{2u-r-1} -  \sum_{u-r\leq i \leq u-1}a_i a_{2u-r-1-i}\Big),
\end{align*}
so \eqref{eq:ell coeff equation} holds for $\ell = 2u-r-1$. By induction we find
\begin{align}\label{eq:inductive defn of coeff aj}
a_j = \frac{1}{2}\Big(q_{u+j} - \sum_{j+1\leq i \leq u-1}a_i a_{u-i+j} \Big)
\end{align}
for $0\leq j \leq u-1$.

It follows readily from \eqref{eq:inductive defn of coeff aj} and induction that the denominator of $a_i$ divides $2^{2(u-i)-1}$, $0\leq i \leq u-1$ (here we use that the $q_i$ are integers), and therefore $2^{2(u-i)}a_i = 4^{u-i}a_i \in \mathbb{Z}$ for $0\leq i \leq u$. 

If we rewrite $j=u-k$ in \eqref{eq:inductive defn of coeff aj} and use the upper bound for $q_j$ then we deduce
\begin{align}\label{eq:inequality for a coeffs}
|a_{u-k}| &\leq \frac{1}{2}\bigg({{2u} \choose k} J^k + \sum_{u-k+1 \leq i \leq u-1} |a_{u-(u-i)}| |a_{u-(i+k-u)}| \bigg).
\end{align}
We claim that $|a_{u-k}| \leq (kuJ)^k$ for $1\leq k \leq u$. Taking $k=1$ in \eqref{eq:inequality for a coeffs} gives $|a_{u-1}| \leq \frac{1}{2}{{2u}\choose 1} J = uJ$, so the claim is true for $k=1$. When we take $k=2$ the sum over $i$ in \eqref{eq:inequality for a coeffs} contains the single term $i=u-1$, so the inequality is $|a_{u-2}| \leq \frac{1}{2}\left({{2u} \choose 2} J^2 + |a_{u-1}|^2\right) \leq \frac{
3}{2}u^2J^2$, where we have used the bound for $|a_{u-1}|$. Hence the claim holds for $k=2$ since $3/2\leq 2^2$. 

Now assume $k\geq 3$, and assume that the claim holds for all $|a_{u-j}|$ with $j < k$. Then by \eqref{eq:inequality for a coeffs} we have
\begin{align*}
|a_{u-k}| &\leq \frac{1}{2}\bigg({{2u} \choose k}J^k + \sum_{u-k+1 \leq i \leq u-1} (u-i)^{u-i}(uJ)^{u-i} (i+k-u)^{i+k-u}(uJ)^{i+k-u} \bigg) \\
&\leq \frac{(uJ)^k}{2}\bigg(\frac{2^k}{k!} + \sum_{1\leq n \leq k-1} n^n (k-n)^{k-n}\bigg),
\end{align*}
where in going from the first line to the second we have changed variables. We note that by symmetry
\begin{align*}
\sum_{1\leq n \leq k-1} n^n (k-n)^{k-n} &\leq 2k^{k-1} + 2\sum_{2\leq n \leq k/2}n^n (k-n)^{k-n} \\ 
&= 2k^{k-1} + 2k^k\sum_{2\leq n \leq k/2}\left( \frac{n}{k}\right)^n \Big(\frac{k-n}{k} \Big)^{k-n} \\
&\leq 2k^{k-1} + 2k^k\sum_{n\geq 2} 2^{-n} = k^k\Big(1 + \frac{2}{k} \Big),
\end{align*}
and therefore
\begin{align*}
\frac{1}{2}\Big(\frac{2^k}{k!} + \sum_{1\leq n \leq k-1} n^n (k-n)^{k-n}\Big) &\leq k^k \Big(\frac{1}{2} + \frac{1}{k} + \frac{2^{k-1}}{k! k^k}\Big).
\end{align*}
Since
\begin{align*}
\frac{1}{2} + \frac{1}{k} + \frac{2^{k-1}}{k! k^k} &\leq \frac{1}{2} + \frac{1}{k} + \frac{1}{2}\Big(\frac{2e}{k^2} \Big)^k\leq 0.95 \leq 1
\end{align*}
for $k\geq 3$ this completes the proof of the claim.

It remains to prove the upper bound on the coefficients of $g$. Since $g = P-f^2$ and by construction $\deg(g) \leq u-1$ we have
\begin{align*}
g(x) = \sum_{i=0}^{u-1}b_ix^i=\sum_{i=0}^{u-1} x^i \Big(q_i - \sum_{0\leq \ell \leq i} a_\ell a_{i-\ell} \Big).
\end{align*}
By the triangle inequality and the bounds for $q_j,a_j$ we obtain
\begin{align*}
|b_i| \leq {{2u} \choose i} J^{2u-i} + (uJ)^{2u-i}\sum_{0\leq \ell \leq i} (u-\ell)^{u-\ell}(u-i+\ell)^{u-i+\ell}.
\end{align*}
We multiply and divide by $u^{2u-i}$ to deduce
\begin{align*}
\sum_{0\leq \ell \leq i} (u-\ell)^{u-\ell}(u-i+\ell)^{u-i+\ell} &= u^{2u-i} \sum_{0\leq \ell \leq i} \Big(1 - \frac{\ell}{u} \Big)^{u-\ell}\Big(1 - \frac{i-\ell}{u}\Big)^{u-i+\ell} \\
&\leq u^{2u-i}\sum_{0\leq \ell \leq u-1} \Big(1 - \frac{\ell}{u} \Big)^{u-\ell}.
\end{align*}
This latter sum over $\ell$ is bounded above by an absolute constant independent of $u$. The contribution from $0\leq \ell \leq u/2$ is
\begin{align*}
&\leq \sum_{0\leq \ell \leq u/2} \Big(1 - \frac{\ell}{u} \Big)^{u/2} \leq \sum_{0\leq \ell \leq u/2}\exp (-\ell/2)\leq \frac{e^{1/2}}{e^{1/2}-1},
\end{align*}
and the contribution from $u/2 < \ell \leq u-1$ is
\begin{align*}
&\leq \sum_{u/2 < \ell \leq u-1} 2^{-(u-\ell)}\leq \sum_{j=1}^\infty 2^{-j} = 1.
\end{align*}
Therefore
\begin{align*}
|b_i| &\leq u^{4u-2i}J^{2u-i}\Big(\frac{2^i}{i!}u^{-(4u-3i)} + 1 + \frac{e^{1/2}}{e^{1/2}-1} \Big) \leq 5u^{4u-2i}J^{2u-i},
\end{align*}
the second inequality following since $u\geq 2$.
\end{proof}

With Lemma \ref{lem:write P as f2 plus g} in hand we can prove that any positive integer $n$ which satisfies $y^2=P(n)$ for some $y \in \mathbb{N}$ is efficiently bounded in terms of $P$.

\begin{theorem}[Strong height bound in even degree]\label{thm:good int point bounds for even degree}
Let $P(x) \in \mathbb{Z}[x]$ be defined as in \eqref{eq:defn of poly P}. If $n$ is a positive integer such that $P(n)=y^2$ for some $y \in \mathbb{N}$ then $n\leq 5 (2u)^{4u}J^{2u}$.
\end{theorem}
\begin{proof}
We assume for contradiction that $n > 5 (2u)^{4u}J^{2u}$ and $P(n) = y^2$. Let $f(x)$ and $g(x)$ be the polynomials of Lemma \ref{lem:write P as f2 plus g}. We note first that $g(x)$ is not the zero polynomial, since otherwise $P(x)=f(x)^2$ would have repeated roots but this contradicts the definition of $P$.

We claim that $g(n) \neq 0$. Since $g(x) \neq 0$ there is a least $0\leq k \leq u-1$ such that $b_k \neq 0$. By Lemma \ref{lem:write P as f2 plus g} we see that $16^u g(x) \in \mathbb{Z}[x]$. If $g(n)=0$ then $x-n$ divides $g(x)$ and therefore $x-n$ divides
\begin{align*}
\sum_{0\leq j \leq u-k-1}b_{k+j} 16^u x^j \in \mathbb{Z}[x].
\end{align*}
Then $n$ divides $16^u|b_k|$, but $16^u |b_k|$ is a nonzero integer of size $\leq 5 (2u)^{4u}J^{2u} < n$, so $n$ does not divide $16^u |b_k|$ and $g(n) \neq 0$.

We next need an upper bound on $|g(n)|$. By the triangle inequality and Lemma \ref{lem:write P as f2 plus g}
\begin{align*}
|g(n)| &\leq 5\sum_{i=0}^{u-1} n^i u^{4u-2i}J^{2u-i} = 5n^{u-1}u^{2u+2}J^{u+1} \sum_{\ell=0}^{u-1} \Big(\frac{u^2 J}{n} \Big)^\ell \leq \frac{5}{1 - u^2J/n} n^{u-1}u^{2u+2}J^{u+1} \\
&\leq 6n^{u-1}u^{2u+2}J^{u+1},
\end{align*}
the sum over $\ell$ coming from a change of variables $i = u-1-\ell$ and the last inequality following since $n > (2u)^{4u} J$ and $u\geq 2$.

Lastly, we need a lower bound on $f(n)$. From Lemma \ref{lem:write P as f2 plus g} we obtain
\begin{align*}
f(n) &\geq n^u - \sum_{i=0}^{u-1} (u-i)^{u-i} (uJ)^{u-i} n^i\geq n^u - u^2 Jn^{u-1} \sum_{\ell=0}^{u-1} \Big(\frac{u^2 J}{n} \Big)^\ell\geq \frac{1}{2}n^u,
\end{align*}
say, since $n > (2u)^{4u} J$ and $u\geq 2$.

From the upper bound for $|g(n)|$, the lower bound for $f(n)$, and the lower bound for $n$ we have 
\begin{align*}
\frac{|g(n)|}{f(n)} &\leq 12\cdot 16^{-u} u^{-2u} J^{1-u} \leq \frac{1}{100},
\end{align*}
say, since $u\geq 2$. Recall now that we assume $y^2 = P(n) = f(n)^2\big (1 + \frac{g(n)}{f(n)^2}\big)$. Taking square roots gives $y = f(n)\big(1 + \frac{g(n)}{f(n)^2} \big)^{1/2}$, where we recall that $y$ is a positive integer. For any real $y$ with $|y| \leq 3/4$, say, we have
\begin{align*}
1 + \frac{y}{2} - \frac{y^2}{4} \leq \sqrt{1+y} \leq 1 + \frac{y}{2},
\end{align*}
and therefore
\begin{align*}
\bigg|y - f(n) - \frac{g(n)}{2f(n)} \bigg| \leq \frac{g(n)^2}{4f(n)^3}.
\end{align*}
We multiply through by $4^u$ to obtain
\begin{align*}
\bigg|4^uy - 4^uf(n) - \frac{4^ug(n)}{2f(n)} \bigg| \leq \frac{4^ug(n)^2}{4f(n)^3},
\end{align*}
where we note that $4^u f(n) \in \mathbb{Z}$ so $4^uy - 4^uf(n) \in \mathbb{Z}$. It follows that there is an integer in the interval
\begin{align*}
\bigg[\frac{4^ug(n)}{2f(n)}\Big(1 - \frac{|g(n)|}{2f(n)^2} \Big),\frac{4^ug(n)}{2f(n)}\Big(1 + \frac{|g(n)|}{2f(n)^2} \Big) \bigg] \subset \bigg[\frac{3}{8}\frac{4^ug(n)}{f(n)},\frac{5}{8}\frac{4^ug(n)}{f(n)} \bigg],
\end{align*}
say, the inclusion following from easy estimations with the upper bound for $|g(n)|$ and the lower bound for $f(n)$, along with the fact that $u\geq 2$. However,
\begin{align*}
\frac{4^u|g(n)|}{f(n)} &\leq \frac{12\cdot 4^u u^{2u+2} J^{u+1}}{n} < \frac{12\cdot 4^u u^{2u+2} J^{u+1}}{5 (2u)^{4u}J^{2u}} \leq \frac{1}{40}
\end{align*}
since $u\geq 2$, so the only possible integer in the interval $\big[\frac{3}{8}\frac{4^ug(n)}{f(n)},\frac{5}{8}\frac{4^ug(n)}{f(n)} \big]$ is zero. But $g(n) \neq 0$, so zero is not contained in this interval.
\end{proof}

\bibliographystyle{amsalpha}

\begin{thebibliography}{9}
\bibitem{BEG2013} A. B\'erczes, J.-H. Evertse, K. Gy\H{o}ry, \emph{Effective results for hyper- and superelliptic equations over number fields}, Publ. Math. Debrecen \textbf{82} (2013), 727--756.
\bibitem{CGPT2012} E. Croot, A. Granville, R. Pemantle, P. Tetali, \emph{On sharp transitions in making squares}, Ann. of Math. \textbf{175} (2012), 1507--1550.
\bibitem{E1991} N. D. Elkies, {\it ABC implies Mordell}, Internat. Math. Res. Notices \textbf{7} (1991), 99--109.
\bibitem{EG1976} P. Erd\H{o}s, R. L. Graham, \emph{On products of factorials}, Bull. Inst. Math. Acad. Sinica 4 (1976), 337--355.
\bibitem{ES1992} P. Erd\H{o}s, J. L. Selfridge, \emph{Getting a square deal, \#6655}, Amer. Math. Monthly \textbf{99} (1992), 791--794.
\bibitem{Fal1983} G. Faltings, \emph{Endlichkeitss\"atze f\"ur abelsche Variet\"aten \"uber Zahlk\"orpern}, Invent. Math. 73 (1983), 349--366.
\bibitem{GS2001} A. Granville, J. L. Selfridge, \emph{Product of integers in an interval, modulo squares}, Electron. J. Combin. \textbf{8} (2001), Research Paper 5, 12 pp.
\bibitem{Guy2004} R. Guy, \emph{Unsolved problems in number theory}. Third edition. Problem Books in Mathematics. Springer-Verlag, New York, 2004.
\bibitem{Hil1986} A. Hildebrand, \emph{On the number of positive integers $\leq x$ and free of prime factors $>y$}, J. Number Theory \textbf{22} (1986), 289--307.
\bibitem{HT1993} A. Hildebrand, G. Tenenbaum, \emph{Integers without large prime factors}, J. Th\'eor. Nombres Bordeaux \textbf{5} (1993), 411--484.
\bibitem{Lang1983} S. Lang, \emph{Conjectured Diophantine estimates on elliptic curves}. Arithmetic and geometry, Vol. I, 155--171, Progr. Math., 35, Birkhäuser Boston, Boston, MA, 1983.
\bibitem{L1993}
M. Langevin, {\it Cas d'\'egalit\'e pour le th\'eor\`eme de Mason et applications de la conjecture (abc)},
C. R. Acad. Sci. Paris Ser. I Math. \textbf{317} (1993), 441--444.
\bibitem{Le1995} M. H. Le, \emph{A note on the integer solutions of hyperelliptic equations}, Colloq. Math. \textbf{68} (1995), 171--177.
\bibitem{Mas2016} D. Masser, \emph{Auxiliary polynomials in number theory}. Cambridge Tracts in Mathematics \textbf{207}, Cambridge University Press, Cambridge, 2016.
\bibitem{Pom1996} C. Pomerance, \emph{A tale of two sieves}, Notices Amer. Math. Soc. 43 (1996), 1473--1485.
\bibitem{Sta2016} K. Stange, \emph{Integral points on elliptic curves and explicit valuations of division polynomials}, Canad. J. Math. \textbf{68} (2016), 1120--1158.
\bibitem{Sza2002} L. Szalay, \emph{Superelliptic equations of the form {$y^p=x^{kp}+a_{kp-1}x^{kp-1}+\dots+a_0$}}, Bull. Greek Math. Soc. \textbf{46} (2002), 23--33.

\end{thebibliography}

\end{document}